\documentclass{amsart}
\usepackage{amssymb,amsmath,amsthm,amsfonts,amscd,amstext}
\usepackage[mathscr]{eucal}
\usepackage{array}
\usepackage{graphicx}
\usepackage{mathtools}
\usepackage{geometry}
\usepackage{cite}
\newtheorem{lem}{Lemma}[section]
\newtheorem{prop}[lem]{Proposition}
\newtheorem{thm}[lem]{Theorem}
\newtheorem{cor}[lem]{Corollary}
\newtheorem{conj}[lem]{Conjecture}
\newtheorem{df}{Definition}[section]
\newtheorem{rem}[lem]{Remark}

\newtheorem*{thm*}{Theorem}
\newtheorem*{prop*}{Proposition}

\usepackage{mathrsfs}

\newcommand{\Z}{\mathbb Z}
\newcommand{\C}{\mathbb C}

\begin{document}
\title{Spectral analysis of  transition operators, 
Automata groups \\
and translation in BBS}
\author[T.~Kato]{Tsuyoshi Kato}
\address{Department of Mathematics, Graduate School of Science, Kyoto
University, Sakyo-ku, Kyoto 606--8502, Japan}
\email{tkato@math.kyoto-u.ac.jp}
\author[S.~Tsujimoto]{Satoshi Tsujimoto}
\address{Department of Applied Mathematics and Physics, Graduate School
of Informatics, Kyoto University, Sakyo-ku, Kyoto 606--8501, Japan}
\email{tujimoto@i.kyoto-u.ac.jp}
\author[A.~Zuk]{Andrzej Zuk}
\address{Institut de Mathematiques, Universite Paris 7, 13 rue Albert Einstein, 75013 Paris, France}
\email{andrzej.zuk@imj-prg.fr}
\begin{abstract}
We give  the automata which describe  time evolution rules of the box-ball system (BBS)
with a carrier. It can be shown by use of tropical geometry, 
such systems are ultradiscrete analogues of KdV equation.
 We discuss their relation with the lamplighter group generated by an automaton.
We present  spectral analysis of the stochastic matrices induced by these automata, and verify their spectral coincidence.

\end{abstract}

\maketitle


\section{Introduction}
From the view point of dynamical systems, {\em automata} constitute 
semi-group actions on trees which play the essential roles in 
 two different subjects, where one is theory of {\em automata groups} 
 and the other is {\em discrete integrable systems}.
 
 Both subjects have been developed from the point of view of dynamical scale transform
 called {\em tropical geometry}\cite{litvinov_maslov,mikhalkin,viro} or {\em ultradiscretization}\cite{tokitaka}
 (they are essentially the same but the original sources have been different,
 where the former arose in real algebraic geometry and the latter from 
 discretization of integrable systems).
It provides with a correspondence between automata and real rational dynamics, which 
by taking scaling limits of parameters,
 allows us to study two dynamical systems at the same time, 
whose dynamical natures are very 
different from each other.
Particularly it  eliminates detailed activities in rational dynamics  and extracts framework of their structure
in automata, which 
allows us to induce some uniform analytic estimates\cite{katotsuji}.

From the computational interests, many of the integrable systems have been 
discretized. In particular KdV equation 
\begin{eqnarray}
 \frac{\partial u}{\partial {t}}+ 6u\frac{\partial u}{\partial {x}}+ \frac{\partial^3 u}{\partial {x}^3}= 0
\end{eqnarray}
is a fundamental equation in the integrable systems,
and its discretized equation has been extensively studied \cite{hirota,hirotsuji}, 
as a rational dynamical systems. In \cite{tokitaka,HT:ukdv}, tropical transform has been applied
 to the discrete KdV equation
\begin{eqnarray}
  \label{eq:dlKdV}
  \dfrac{1}{u_{n+1}^{(t+1)}} - \dfrac{1}{u_n^{(t)}}
 = \frac{\delta}{1+\delta}\left(u_{n+1}^{(t)} - u_{n}^{(t+1)}\right),
\end{eqnarray}
and the ultradiscrete KdV equation 
\begin{eqnarray}
B^{(t+1)}_n = \min(1-B^{(t)}_n,\sum_{j=-\infty}^{n-1}(B^{(t)}_{j} - B^{(t+1)}_{j})) \label{s10}
\end{eqnarray}
is obtained, which is the so-called {\it box-ball systems} (BBS)\cite{takahashi1990}.
We verify that BBS is described by automata diagram:
\begin{center}
 \includegraphics[height=1.5cm]{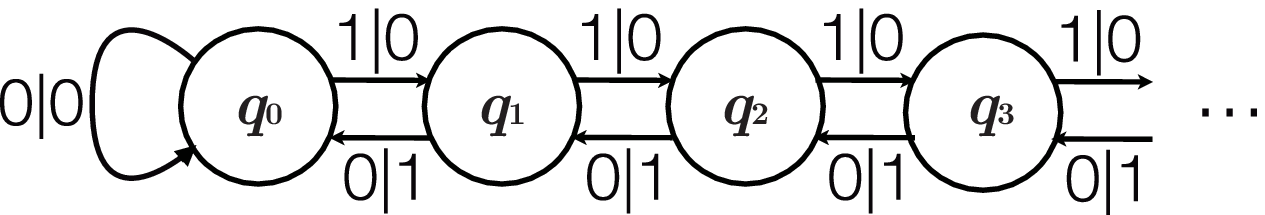}
\end{center}
which is given by a direct limit of the Mealy type automata BBS$_k$ for $k \geq 1$, that are the carrier capacity extensions of BBS (see Section \ref{sec:bbsc}).
Moreover each BBS$_k$  is described by automata diagram (Lemma \ref{lem:3.1}).

Rational dynamics can be regarded as approximations of the corresponding evolutional systems
in partial differential equations\cite{kato3}.
From the view point of dynamical scaling limits,
automata can be regarded as frame-dynamics which play the roles of underlying mechanics 
for PDE\cite{kato2}.
From dynamical view point, distribution of orbits can be measured by probability approach, which  
 is a quite fundamental method. So far there has been little done on study of BBS from probability aspect.
On the other hand study of random walk on automata group has been extensively developed.
So our basic and general question is, whether the frame dynamics of integrable systems
share their structural similarity with geometric properties of automata  groups.
It would give us much deeper understanding of dynamical structure of BBS.

In \cite{kato1}, we have verified that 
the automaton is recursive if and only if the associated rational dynamics
is quasi-recursive.
Quasi-recursiveness represents `almost' recursive which differs from periodic 
 within uniform estimates
 independent of the choice of initial values. 
As an extension of the above property, 
we have   applied tropical geometry to
theory of automata groups  
 to analyze global behaviour of real rational dynamical systems.
A discrete group is called an {\em automata group}, if it is realized by  actions on the rooted trees,
which are represented by a Mealy automaton. 
The automata group is a quite important class in group theory,
which 
 have given answers to many important questions.
Of particular interests for us are, counter-example to the  Milnor's conjecture,
 solution to the Burnside problem on
the existence of finitely generated infinite torsion groups, 
non-uniform exponential growth groups, etc.
These applications are described in \cite{zuk1} and \cite{zuk2}.
As an application  of tropical geometry to the construction of the Burnside group,
we have verified that there 
exists a rational dynamical systems of Mealy type 
which satisfy infinite quasi-recursiveness
 \cite{kato4}.
 This  property again allows error from recursiveness   
which corresponds to infinite torsion, while rationality corresponds to finite generation.

In case of finitely generated 
groups one can consider as the space $l^2$ functions on groups and as the operator 
the sum of translations by chosen generators and their inverses. The study of spectra
of such operators was initiated by Kesten and the normalized operators are called 
random walk or {\em transition operators} \cite{woess}.

In general it is a very difficult problem to compute spectra of these operators. 
Some important progresses have been achieved by studying
different approximations of such operators using the representations of the group, in particular their actions on finite sets. For instance in \cite{BVZ} the spectrum of the random walk operator on the Heisenberg group was computed using approximations Harper operators via theory of
rotational algebras. In case of groups generated by automata one can study their 
actions on finite sequences. The simplest case when one obtains
an interesting spectral information is the automaton on two states which generated the lamplighter group. All other two state automata lead to very elementary cases.
In the case of BBS we do not deal with invertible transformations which would define groups. However we can still define the operators similar to random walk operators and consider their
action on finite sequences. This enables us to compute the limit spectral measures for such sequences as was done for automata in \cite{grizuk}.

Even though both BBS and automata group are constructed from Mealy automata,
their scopes are quite different. As a result, one finds quite different characteristics
from each other. It would be of particular  interest for us to combine such properties via
dynamical study of Mealy automata.
We want to investigate BBS systems  via spectra of some operators associated to them, as
it is common in non-commutative geometry.

Recall that both the lamplighter group and BBS act on the rooted binary tree $T$.
For convenience let us describe its action in the case of the lamplighter group (see Section \ref{sec:lamp}).
Each state acts on the binary tree, and in this case there are two actions ${\bf A}_{q_0}$ 
and ${\bf A}_{q_1}$
corresponding to the state set $\{q_0,q_1\}$. It is defined inductively on each level set.
The actions on the second level set are depicted as follows:
\begin{center}
  \includegraphics[height=4cm]{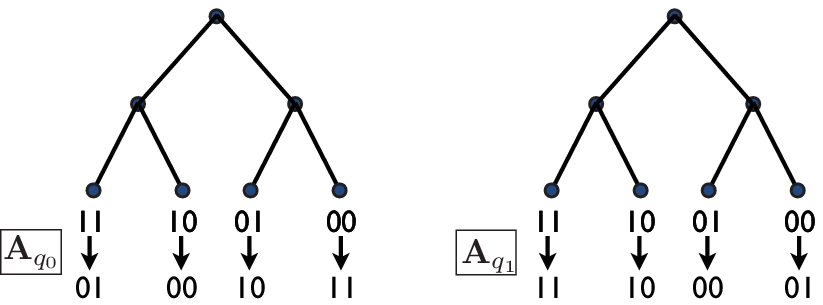}
\end{center}

In general  transition operators of automata  can be given by filtration of finite rank matrices
with respect to the restriction of the action on each level set.
It was known that these matrices of the lamplighter automaton all satisfy stochastic property.
Such property is quite important for automata groups in relation to random walk on automata groups.
In comparison to BBS automata, 
we have verified that the  same  property holds  for all $k \geq 1$.

Let $\{M_k^{(n)}\}_{n \geq 1} $ be the filtration of the transition operators for 
BBS$_k$. They are defined in Section \ref{sec:bbs_translation}
by $\frac{1}{2k+2}\sum_{i=0}^{k}(a_i^{\!(n)} + a_i^{\!(n)*})$,
where $a_i^{(n)}$ are the restrictions of the representation matrices on the level $n$ set 
and the superscript $*$ denotes the transpose of the matrix.

\begin{prop*}[Proposition \ref{prop61}]
The matrix $M_k^{(n)}$ is double stochastic for all $k \geq 1$, $n \geq 0$, i.e. 
the sum of each row and each  column is equal to $1$.
\end{prop*}
It is known that there is an example of Mealy automata whose
transition operators do not satisfy stochastic property.

The simplest case BBS $k=1$ 
(BBS translation)
satisfies 
rather trivial  behavior from dynamics view points, since it is translation.
However 
the above proposition  would suggest that BBS translation is closely related to the cases for $k \geq 2$.
Concerning BBS translation, 
we have discovered  non trivial phenomena from spectral analysis view point.
\begin{thm*}[Theorem \ref{thm71}]
{\normalfont (i)} The spectra of the transition operators coincide with each other
between the lamplighter group as an  automata group 
and the BBS translation.
It is totally discrete and dense in $[-1,1]$.

\end{thm*}

Because
 the eigenvalue distributions coincide, we may expect that 
 both transition operators are mutually conjugate by some orthogonal matrices.
 Actually we verify that it certainly holds. Moreover
it would be natural to ask whether the conjugation might be chosen from tree automorphisms.

We have the negative answer:
\begin{prop*}[Proposition \ref{lem8.5}]
There are  no automorphisms of $T$ which conjugate 
between $M_L^{(n)}$ and $M_B^{(n)}$.
\end{prop*}

On the other hand, one might still ask whether it comes from permutations, 
or 
from an automorphism of the one sided shift.
We have the affirmative answer, which gives the complete answer 
to
the conjugations.

Let $M_B^{(n)} , M_L^{(n)}  \in \mbox{Mat}(2^n \times 2^n; {\mathbb R})$.
Let us denote the set of indices as $I_n=\{0,1,\ldots,2^n-1\}$.
\begin{thm*}[Theorem \ref{thm:BBS to LL}]
 There exists a permutation   matrix $\sigma_n$, such that
\begin{align*}
 \sigma_n^* M_B^{(n)} \sigma_n =M_L^{(n)}
\end{align*}
holds.
We have the explicit recurrence formulas for $\sigma_n$ which involves the  Sierpinski gasket pattern.
\end{thm*}

In the cases for $k \ge 2$, we present numerical computations of the spectral distributions and observe that there exist structural similarities in the distributions of eigenvalues between the BBS$_k$  and the lamplighter group. 
Based on the observations, we give conjectures (Conjecture \ref{conj1} and \ref{conj2}).

If one reduces an integrable system to an automaton  by extracting its dynamical framework,
then it should posses high symmetry, which will have some structural similarity with 
finitely generated groups. 
It would be interesting to investigate further coincidence between 
spectra of  automata 
associated to integrable system
and the one  associated to automata groups.

\section{Automata groups}
 An
 {\em automaton} is defined  by finite rules which  can create quite complicated state dynamics 
 over the sequences of alphabets.

Let $Q$ and $S$ be finite sets, and consider the set of all infinite sequences:
$$S^{\mathbb N}=\{(s_0,s_1, \dots): s_i \in S\}.$$

A  {\em Mealy automaton}  ${\bf A}$  is given by  a pair of functions:
$$\varphi: Q \times S \to Q, \quad
\psi: Q \times S \to S$$
which  gives rise to the state dynamics on $S^{\mathbb N}$ as follows.
Let us choose any $q\in Q$ and $\bar{s}=(s_0,s_1, \dots) \in S^{\mathbb N}$.
Then:
$${ \bf A}_q: S^{\mathbb N} \to S^{\mathbb N}$$
 ${\bf A}_q (\bar{s})=(s_0',s_1', \dots)$ is determined 
 inductively by:
  $$s_i'= \psi(q_i,s_i), \quad
 q_{i+1} = \varphi(q_i, s_i) \quad (q_0=q).$$
Besides the dynamics over $S^{\mathbb N}$, 
the change of the state sets play important roles in a hidden dynamics.

 Any sequences
 $\bar{q}^j =(q^0,\dots, q^j) \in Q^{j+1}$
 give dynamics by compositions:
 $${\bf A}_{\bar{q}^j} = {\bf A}_{q^j} \circ \dots \circ {\bf A}_{q^0} : S^{\mathbb N} \to S^{\mathbb N}.$$
It can happen  that  different automata give the same state dynamics.
In such a case, the dynamics of ${\bf A}_q$ are the same, but the systems of change of state sets 
can be very different.
Such two automata are called {\em equivalent}.

Suppose 
$$\psi: (q, \quad) : S \to S$$ are permutations for all $q \in Q$.
If we identify $S^{\mathbb N}$ with the rooted tree, then the Mealy dynamics give the group actions on the tree,
since the actions can be restricted  level-setwisely.
 The group generated by these states is called the {\em automata group}
given by the automaton $(\varphi,\psi)$.

Next we introduce the diagram expression of the automaton ${\bf A}$ defined via the quadruple $(Q,S,\varphi,\psi)$.
In the diagram each vertex corresponds to a state $q \in Q$.
When $\varphi(q,i)=r$ and $ \psi(q,i)=j$, the vertex $q$ is connected to the vertex $r$ with the directional arrow equipped with the pair of input and output strings, $i \mid j$.

\subsection{Lamplighter group}\label{sec:lamp}
The lamplighter group:
$$(\oplus_{\mathbb Z} {\mathbb Z}_2) \rtimes {\mathbb Z}$$
 is generated by canonical generators, which are 
$v$,  one copy of ${\mathbb Z}_2$,
and $u$,   the generator of ${\mathbb Z}$.

The corresponding automaton  can be represented by  the following 
diagram:
\begin{center}
 \includegraphics[height=1.5cm]{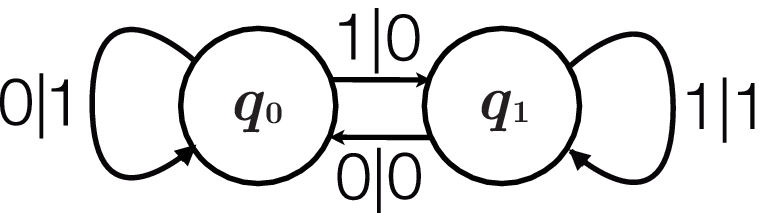}
\end{center}
which shows that the quadruple $(Q,S,\varphi,\psi)$ of the lamplighter group as an automata group is given by
\begin{eqnarray*}
&& Q=\{q_0,q_1\}, \quad S=\{0,1\}, \\
&& \varphi(q_0,0)=0,\quad \varphi(q_0,1)=1, \quad \varphi(q_1,0)=0, \quad \varphi(q_1,1)=1,\\
&& \psi(q_0,0)=1,\quad \psi(q_0,1)=0, \quad \psi(q_1,0)=0, \quad \psi(q_1,1)=1.
\end{eqnarray*}

For example, we give actions of the lamplighter automata ${\bf A}_{q_0}$ and ${\bf A}_{q_1}$  as follows:
\begin{eqnarray*}
 &&{\bf A}_{q_0} (0011101100000\cdots)
  = 1101100101111\cdots\\
 &&{\bf A}_{q_1} (0011101100000\cdots)
  = 0101100101111\cdots
\end{eqnarray*}

Let $a_i$ be the infinite matrix representations of ${\bf A}_{q_i}$ for $i=0,1$.
They decompose into $2$ by $2$ matrices with operator valued entries,
with respect to 
$$S^{\mathbb N}= S^{\mathbb N}_0
 \sqcup S^{\mathbb N}_1$$
where $S^{\mathbb N}_i =\{(i,s_1, \dots): s_i \in S\}$.

In the lamplighter  case, we have two operator recursions \cite{grizuk}
\begin{eqnarray}
\label{lamplighter operators}
 a_{0} 
= \left(\begin{array}{cc}
   0&a_{1} \\   a_{0}&0 \\
        \end{array}\right),
\quad
 a_{1} 
= \left(\begin{array}{cc}
   a_{0} &0\\   0&a_{1}  \\
        \end{array}\right),
\end{eqnarray}
where $a_1^{-1} a_0$ corresponds to $v$
and $a_0$ to $u$.

\section{BBS with carrier capacity}\label{sec:bbsc}

The BBS is one of the ultradiscrete integrable systems. The BBS is composed of an array of infinitely many boxes, finite number of  balls in the boxes, and a carrier of balls. Each box can contain only one ball and the carrier can hold arbitrary number of balls. The evolution rule from time $j$ to time $j+1$  is defined as follows. The carrier moves from left to right and passes each box. When the carrier passes a box containing a ball, the carrier gets the ball; when the carrier passes an empty box, if the carrier holds balls, the carrier puts one ball into the box.

\makeatletter
\newcommand{\BA}{\circle{8}}
\newcommand{\BB}{\circle*{8}}
\newcommand{\BC}{\circle{9}}
\newcommand{\BD}{\circle*{9}}
\makeatother
\begin{figure}[h]
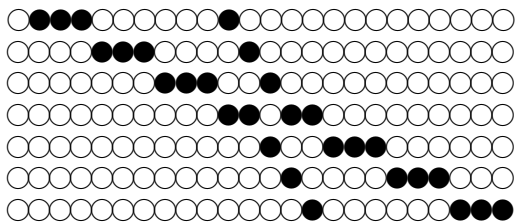

\begin{center}
\noindent
{
\hspace*{-1mm}\BA\BB\BB\BB\BA\BA\BA\BA \BA\BA\BB\BA\BA \BA\BA\BA\BA\BA \BA\BA\BA\BA\BA\BA \\
\BA\BA\BA\BA\BB\BB\BB\BA \BA\BA\BA\BB\BA \BA\BA\BA\BA\BA \BA\BA\BA\BA\BA\BA \\
\BA\BA\BA\BA\BA\BA\BA\BB \BB\BB\BA\BA\BB \BA\BA\BA\BA\BA \BA\BA\BA\BA\BA\BA \\
\BA\BA\BA\BA\BA\BA\BA\BA \BA\BA\BB\BB\BA \BB\BB\BA\BA\BA \BA\BA\BA\BA\BA\BA \\
\BA\BA\BA\BA\BA\BA\BA\BA \BA\BA\BA\BA\BB \BA\BA\BB\BB\BB \BA\BA\BA\BA\BA\BA \\
\BA\BA\BA\BA\BA\BA\BA\BA \BA\BA\BA\BA\BA \BB\BA\BA\BA\BA \BB\BB\BB\BA\BA\BA \\
\BA\BA\BA\BA\BA\BA\BA\BA \BA\BA\BA\BA\BA \BA\BB\BA\BA\BA \BA\BA\BA\BB\BB\BB \\
}
\end{center}
\caption[11]{A two-soliton interaction of the BBS} 
\end{figure}

Let us describe BBS with carrier capacity $k$ \cite{MatsuTaka}. In this case the carrier can hold at most $k$
balls. The only difference with the previous situation is that when the carrier holds $k$ balls and 
passes a box containing a ball, the carrier does nothing. 

Similar to the case of KdV,  the BBS with carrier capacity $k$ can be obtained from the discrete modified KdV equation \cite{TH}
\begin{align}
\label{disc-mKdV}
v_{n+1}^{(t+1)}
\dfrac{(1+\alpha) v_{n}^{(t+1)} + \delta }{(1+\delta) v_{n}^{(t+1)} + \alpha}
=v_n^{(t)}
\dfrac{(1+\alpha) v_{n+1}^{(t)} + \delta }{(1+\delta) v_{n+1}^{(t)} + \alpha},
\end{align}
where $\alpha,\delta$ are constants,
which reduces to the modified KdV equation
\begin{align*}
  \frac{\partial {v}}{\partial t} + 6\beta {v}^2\frac{\partial {v}}{\partial x}+ \dfrac{1}{4\beta} \frac{\partial^3 {v}}{\partial {x}^3}=0,
\end{align*}
where $\beta$ is a constant.
The BBS with carrier capacity $k$
is presented by
\begin{align*}
\widetilde B_n^{(t+1)} 
= \min\left(1-\widetilde B_n^{(t)}, \sum_{j=-\infty}^{n-1}(\widetilde B_j^{(t)}-\widetilde B_j^{(t+1)})\right)
 +\max\left(0,\sum_{j=-\infty}^{n}(\widetilde B_j^{(t)}-\widetilde B_{j-1}^{(t+1)})-k\right).
\end{align*}

\begin{lem}\label{lem:3.1}
The diagram expression of the BBS with carrier capacity $k$ is given by
\begin{center}
 \includegraphics[height=1.5cm]{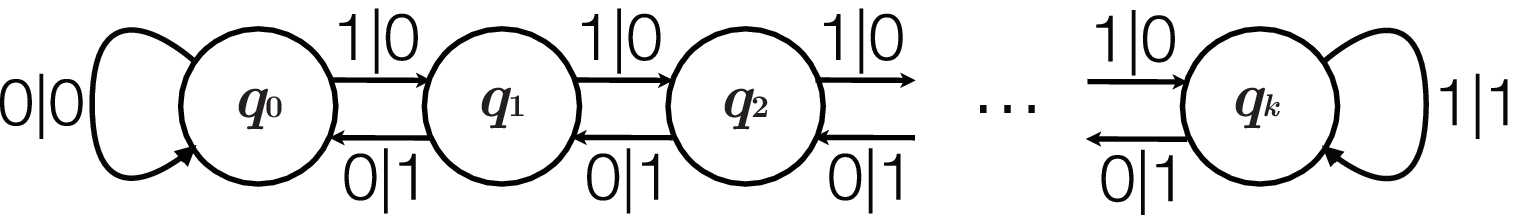}
\end{center}

The (simple) BBS is obtained as the limiting case of the above automaton
with $k \to \infty$.
\end{lem}

\begin{proof}
The state $q_i$ corresponds to the situation when the carrier holds $i$ balls. 
Thus we start at the state $q_0$. 
If we have 1 as the input we go from the state $q_i$  to $q_{i+1}$ if 
$i < k$ and we change 1 to 0. This corresponds to the fact the carrier 
picks up the ball if the number of balls it 
already holds is $i < k$.
If we have 0 as the input we go from the state $q_i$  to $q_{i-1}$ if 
$i >0$ and change 1 to 0. This corresponds to the fact the carrier 
puts the ball if the number of balls it 
already holds is at least 1. It remains to check the situation for $q_0$ with the input 0 and
for $q_k$ with the input 1. The first one corresponds to the carrier with 0 balls passing an
empty box (it does nothing and still holds no balls) and the last one to  the carrier with $k$ balls passing a box with a ball (it does nothing and still holds $k$ balls). 
\end{proof}

\subsection{BBS translation (carrier capacity $k=1$)}
\label{sec:bbs_translation}
The BBS translation can be represented by 
\begin{center}
 \includegraphics[height=1.5cm]{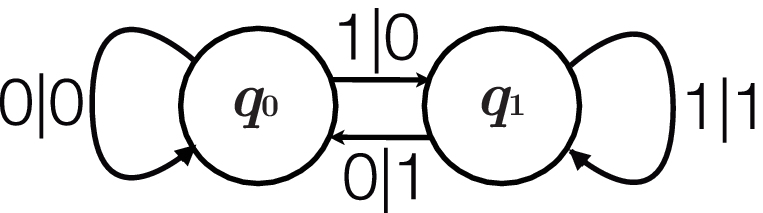}
\end{center}
For example, we give actions of the BBS translation ${\bf A}_{q_0}$ and ${\bf A}_{q_1}$  as follows:
\begin{eqnarray*}
 &&{\bf A}_{q_0} (0011101100000\cdots)
  = 0001110110000\cdots\\
 &&{\bf A}_{q_1} (0011101100000\cdots)
  = 1001110110000\cdots
\end{eqnarray*}

Let $a_i$ be the infinite matrix representations of ${\bf A}_{q_i}$ for $i=0,1$. 
Then we have two operator recursions
\begin{eqnarray}
\label{bbs operators}
 a_{0} 
= \left(\begin{array}{cc}
   a_{0}&a_{1} \\   0&0 \\
        \end{array}\right),
\quad
 a_{1} 
= \left(\begin{array}{cc}
   0&0\\   a_{0}&a_{1} \\
        \end{array}\right).
\end{eqnarray}

We can describe the action of $a_0$ and $a_1$ on the binary sequences of length $n$ 
by the $2^n \times 2^n$ matrices $a_0^{(n)}$ and $a_1^{(n)}$.
From the definition of our automaton, they satisfy the following recurrence relations:
\begin{eqnarray*}
&&
a_{0}^{(0)}=
a_{1}^{(0)}=1,
\\
&& a^{(n+1)}_{0} 
= \left(\begin{array}{cc}
   a^{(n)}_{0}& a^{(n)}_{1} \\[1mm]   0 & 0  \\
        \end{array}\right),
\quad
 a^{(n+1)}_{1} 
= \left(\begin{array}{cc}
      0 & 0  \\a^{(n)}_{0}& a^{(n)}_{1} \\
        \end{array}\right).
\end{eqnarray*}

In the case when ${\bf A}_{q_j}$ for $j=0,1$ give automorphisms
so that they constitute an automata group, the transition operator 
is given by $M = \frac{1}{4}(a_0+a_0^* + a_1+a_1^*)$
which describes step $1$ random walk. 
In the general case when the actions are not invertible, 
one can still consider the same operators, since the adjoint operators coincide with the inverse ones
for the invertible  case, since they are unitaries.
The
 actions by ${\bf A}_{q_j}$ are always deterministic, while its adjoint ${\bf A}_{q_j}$ 
 are  non-deterministic  in non invertible case. A key observation is that random walk on 
 semi-groups is still possible to formulate and would be quite natural,
  if we allow non deterministic actions and interpret them as probability processes.
In this paper we shall introduce the transition operator over the semi-groups by the same formula.

Let us  consider the filtrations of the  transition operators:
\begin{eqnarray*}
 M_{k=1}^{(n)}=\dfrac{1}{4}\left(a^{(n)}_{0} + a^{(n)*}_{0}+a^{(n)}_{1} + a^{(n)*}_{1}\right).
\end{eqnarray*}

\subsection{BBS with carrier capacity $k=2$}
In analogy to $k=1$ case, for $k=2$, we can consider the following operators.
\begin{center}
  \includegraphics[height=1.5cm]{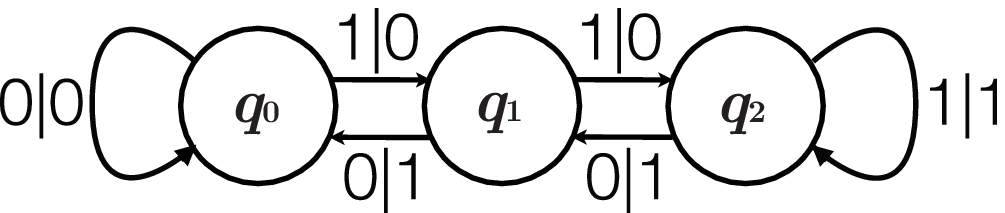}
\end{center} 
For example, we give actions of the BBS$_{k=2}$ ${\bf A}_{q_0}, {\bf A}_{q_1}$ and ${\bf A}_{q_2}$  as follows:
\begin{eqnarray*}
 &&{\bf A}_{q_0} (0011101100000\cdots)
  = 0000110111000\cdots\\
 &&{\bf A}_{q_1} (0011101100000\cdots)
  = 1000110111000\cdots\\
 &&{\bf A}_{q_2} (0011101100000\cdots)
  = 1100110111000\cdots
\end{eqnarray*}
Notice that the state $q_j$ corresponds to the carrier with $j$-number of balls. The action ${\bf A}_{q_j}$ represents the time-evolution of BBS${}_k$ dynamics with carrier with $j$-number of balls as an initial state.

Let $a_i$ be the infinite matrix representations of ${\bf A}_{q_i}$ for $i=0,1,2$. 
Then we have three operator recursions
\begin{eqnarray*}
 a_0 
= \left(\begin{array}{cc}
   a_0&a_1 \\   0&0 \\
        \end{array}\right),
\quad
 a_{1}  
= \left(\begin{array}{cc}
   0&a_2\\   a_0&0 \\
        \end{array}\right),
\quad
 a_2  
= \left(\begin{array}{cc}
   0&0\\   a_1&a_2 \\
        \end{array}\right).
\end{eqnarray*}
The action of $a_0, a_1$ and $a_2$ on the binary sequences of length $n$ can be described by the $2^n \times 2^n$ matrices $a_0^{(n)}, a_1^{(n)}$ and $a_2^{(n)}$ which satisfy the following recurrence relations:
\begin{eqnarray*}
&&
a_{0}^{(0)}=
a_{1}^{(0)}=
a_{2}^{(0)}=1,
\\
&& a^{(n+1)}_{0}  
= \left(\begin{array}{cc}
   a^{(n)}_{0} & a^{(n)}_{1}  \\[1mm]   0 & 0  \\
        \end{array}\right),
\quad
 a^{(n+1)}_{1}
= \left(\begin{array}{cc}
      0 & a^{(n)}_{2}  \\a^{(n)}_{0} & 0 \\
        \end{array}\right),
\quad
 a^{(n+1)}_{2} 
= \left(\begin{array}{cc}
      0 & 0  \\a^{(n)}_{1}& a^{(n)}_{2} \\
        \end{array}\right).
\end{eqnarray*}

The filtrations of the transition operators are given by
\begin{eqnarray*}
 M^{(n)}_{k=2}=\dfrac{1}{6}\left(a^{(n)}_{0} + a^{(n)*}_{0}+a^{(n)}_{1} + a^{(n)*}_{1}+a^{(n)}_{2} + a^{(n)*}_{2}\right).
\end{eqnarray*}
In the next sections, we will verify that these transition operators are stochastic and analyze in detail the spectral properties of the transition operator for $k=1$ theoretically  and $k \ge 2$ numerically.

\section{Stochastic matrices}
Study of countably 
state ergodic Markov chain is an important subject in relation with statistic mechanics.
However because of countably many number of the states, construction of the probability measures
on the path space has not been so developed. 
It follows from Corollary \ref{cor:5.4} below that BBS transition operators $M^{(n)}_k$ give the ergodic Markov chains
over the set of paths $\Omega(n)$ which are given by
$\Omega(n) = \{(w_1,w_2, \dots) \mid w_i \in \{1, \dots, 2^{n}\} \}$
with 
the unique ergodic distributions
$\pi^{(n)} =(\pi_1^{(n)}, \dots , \pi^{(n)}_{2^n})$
\cite{Sinai_book}.
Let $m_k$ be the probability measure on $\Omega(n)$.
One may expect that the family of ergodic Markov chains defined by $\{M_k^{(n)}\}_{n=1}^{\infty}$
can give a countably state Markov chains over the path space:
$$\Omega(\infty) =  
 \{(w_1,w_2, \dots) \mid w_i \in {\mathbb N} \ \}$$
which is expected ergodic at the limit.

Let us verify stochastic property of the transition operators for 
BBS$_{k}$.
We define a sequence of $k+1$ matrices $(a_0^{(n)}, \ldots , a_k^{(n)})$
of dimension $2^n$, for $n = 0, 1 , \ldots$ 
by  the following matrix recursion ($0$ represents here $2^n \times 2^n$ null matrix).

$$ a_0^{(n+1)}  =  \left (
\begin{array}{cc}
a_0^{(n)} &  a_1^{(n)} \\
0  &  0
\end{array}
\right )
$$

For $i=1, \ldots, k-1$

$$ a_i^{(n+1)}  =  \left (
\begin{array}{cc}
0 &  a_{i+1}^{(n)} \\
a_{i-1}^{(n)}  &  0
\end{array}
\right )
$$

and  

$$ a_k^{(n+1)}  =  \left (
\begin{array}{cc}
0 &  0  \\
a_{k-1}^{(n)}  &  a_{k}^{(n)}
\end{array}
\right )
$$

with the initial data $a_i^{(0)} =1$ for all $i = 0, \ldots , k$.

We consider the following $2^n \times 2^n$ matrix $M_k^{(n)}$
$$ 
M_k^{(n)}  =  \frac{1}{2k+2}(a_0^{(n)}  + a_0^{(n)*} + \ldots +  a_k^{(n)}  + a_k^{(n)*} ) .
$$

\begin{prop}\label{prop61}
The matrix $M_k^{(n)}$ is double stochastic for all $k \geq 1$, $n \geq 0$, i.e. 
the sum of each row and each  column is equal to $1$.
\end{prop}

\begin{proof}
The matrix $M_k^{(n)}$ is symmetric and therefore it suffices to prove that the sum of columns 
is constant. 

Clearly the recursive relations for $a_0^{{(n+1)}} , \ldots , a_k^{{(n+1)}}$ show that 
the matrix
we obtain 
from each of them is the matrix with constant sum of columns (equal to $k$).

Thus it is enough to show that 
$ a_0^{(n)*}   + \ldots  +  a_k^{(n)*}$
has constant column sum. Let us prove this by induction. It is clear
for $n=0$. 
Then using recursion formula
$$ a_0^{(n+1)*}   + \cdots  +  a_k^{(n+1)*}
=
 \left (
\begin{array}{cc}
 a_0^{(n)*}  &   a_{k-1}^{(n)*}  + \cdots  +  a_0^{(n)*}  \\
 a_1^{(n)*}  + \cdots  +  a_k^{(n)*}  &  a_k^{(n)*}  
\end{array}
\right )
$$
and thus the sum of the left matrix blocks and right matrix blocks is equal to
$$ a_0^{(n)*}  + \ldots  +  a_k^{(n)*}.$$
Therefore the statement follows by induction.
\end{proof}

\subsection{Spectral computation for BBS translation ($k=1$)}
Stochastic property closely related to random walk on each level set of the binary tree.
On the other hand structure of the random walk heavily depends on their spectral distribution.
First we compute the spectral distribution of the transition operator for $k=1$:
\begin{eqnarray*}
 M_{k=1}^{(n)}=\dfrac{1}{4}\left(a^{(n)}_{0} + a^{(n)*}_{0}+a^{(n)}_{1} + a^{(n)*}_{1}\right).
\end{eqnarray*}

Define the counting spectral measures of
$M^{(n)}_{k}$, 
i.e. $\sigma_k^{(n)} : [0, 1] \to [0,1]$ and for $x \in [0,1]$ by:
$$
 \sigma_k^{(n)}(x) 
= \dfrac{\sharp\left\{\lambda \in \mbox{Sp}(M_k^{(n)})\mid
			     \lambda \le 2(k+1)\cos(\pi x)\right\}}
        {\sharp\left\{\lambda \in \mbox{Sp}(M_k^{(n)})\right\}}.$$ 
Let us denote the multiplicity of eigenvalue $\lambda$ of
$M^{(n)}_k$ 
by 
\begin{eqnarray*}
 m^{(n;k)}(\lambda)=\sharp\{\lambda \in \mbox{Sp}(M_k^{(n)})\}. 
\end{eqnarray*}
In particular we denote the multiplicity of eigenvalue $\cos\left(pq^{-1}\pi\right)$
of $M^{(n)}_k$ by $m_{p,q}^{(n;k)}=
m^{(n;k)}\left(\cos\left(pq^{-1}\pi\right)\right)$.

We consider the case $k=1$ and provide the computation of eigenvalues of $M_{k=1}^{(n)}$.
Our computation on the spectra verify the following:
\begin{thm}\label{thm71}
\begin{eqnarray*}
\mbox{Sp}\left(M_{k=1}^{(n)}\right) =\mbox{Sp}\left(\dfrac{1}{4}\sum_{j=0}^{1}\left( a^{(n)}_{j}+a^{(n)*}_{j}\right)\right)
 = \left\{ 1 \cup \cos\left(\dfrac{p}{q}\pi\right)\bigg|\,
p,q\in {\mathbb N}, 1\le p<q\le n+1 \right\} 
\end{eqnarray*}
If $p$ and $q$ are mutually prime, then the 
multiplicity of eigenvalue $\cos\left(pq^{-1}\pi\right)$,
denoted by $m_{p,q}^{(n;1)}$, is given by
\begin{eqnarray*}
 m_{p,q}^{(n;1)}
 =  \left[2^n \left(\dfrac{2^{-q}-2^{-q\left(\left[\frac{n}{q}\right]+1\right)}}{1-2^{-q}}\right)\right]
\end{eqnarray*}
\end{thm}
In order to simplify the notation we define $a_n = a_0^{(n)}$ and $b_n = a_1^{(n)}$.
\begin{lem}\label{lem63}
For every $n$ 
$$ a_n  a_n^*    +  b_n b_n^*  = 2 \text{Id}_{\,2^n}.$$
\end{lem}
\begin{proof}
We have

$$
a_{n+1}   a_{n+1}^*  =
 \left (
\begin{array}{cc}
  a_n  &   b_n  \\
 0  &  0
\end{array}
\right )
 \left (
\begin{array}{cc}
  a_n^*  &  0  \\
 b_n^*  &  0
\end{array}
\right )
=
 \left (
\begin{array}{cc}
a_n   a_n^*  +   b_n   b_n^*&  0  \\
0 &  0
\end{array}
\right )
$$

$$
b_{n+1}   b_{n+1}^*  =
 \left (
\begin{array}{cc}
 0  &  0  \\
  a_n  &   b_n  
\end{array}
\right )
 \left (
\begin{array}{cc}
0 &   a_n^*   \\
0 &  b_n^*  
\end{array}
\right )
=
 \left (
\begin{array}{cc}
0 &  0  \\
0 &  a_n   a_n^*  +   b_n   b_n^*
\end{array}
\right )
$$
and the statement follows by induction.
\end{proof}

{\em Proof of Theorem $6.1$:}
Let us put:
$$\Phi_n ( \lambda , \mu ) = \det  ( a_n +  a_n^*  + b_n + b_n^*   -  \frac{1}{2} \mu  ( a_n b_n^*  +  b_na_n^*)
- \lambda Id_{2^n}  )$$

Then by  use of Lemma \ref{lem63},  we have the equalities:
\begin{align*}
\Phi_{n+1}  ( \lambda , \mu )  
& =   \det (  a_{n+1} +  a_{n+1}^*  + b_{n+1} + 
b_{n+1}^*   - \frac{1}{2} \mu  ( a_{n+1} b_{n+1}^*  +  b_{n+1} a_{n+1}^*)
- \lambda Id_{2^{n+1}}  )  \\
&  = 
\det  
\left (
\begin{array}{cc}
a_n + a_n^*  - \lambda &   b_n  + a_n^*  -  \frac{1}{2}\mu ( a_n a_n^*  + b_n b_n^* ) \\
a_n + b_n^* - \frac{1}{2}\mu ( a_n a_n^* + b_n b_n^*) &  b_n +  b_n^* - \lambda 
\end{array}
\right ) \\
&  = 
\det  
\left (
\begin{array}{cc}
a_n + a_n^*  - \lambda &   b_n  + a_n^*  -  \mu   \\
a_n + b_n^* - \mu &  b_n +  b_n^* - \lambda 
\end{array}
\right ) \\
&  = 
\det  
\left (
\begin{array}{cc}
a_n -  b_n  - \lambda  +  \mu  &   b_n  + a_n^*  -  \mu  \\
a_n  - b_n  + \lambda - \mu &  b_n +  b_n^* - \lambda 
\end{array}
\right )  \\
&  = 
\det  
\left (
\begin{array}{cc}
2 \mu  -   2 \lambda    &   a_n^* -  b_n^*  -  \mu + \lambda \\
a_n  - b_n  + \lambda - \mu &  b_n +  b_n^* - \lambda 
\end{array}
\right )  
\end{align*}
Using the fact that 
$$  \det 
\left (
\begin{array}{cc}
A  &   B \\
C &  D 
\end{array}
\right )  = \det (AD - CB)$$ 
provided that $A$ commutes with $C$ we get
\begin{align*} \Phi_{n+1}  ( \lambda , \mu )  
& =  
\det (    (2 \mu  -   2 \lambda) ( b_n +  b_n^* - \lambda)  -  (a_n  - b_n  + \lambda - \mu)
(a_n^* -  b_n^* -\mu+ \lambda ) ) \\
& =  
\det (  ( \mu  -   \lambda)  (a_n +  a_n^*  + b_n + b_n^*)   -   \frac{1}{2} 2 ( a_n b_n^*  +  b_na_n^*)
+ (-2  + \lambda^2  -  \mu^2)  Id_{2^n}  ) 
\end{align*}
Therefore
$$\Phi_{n+1} ( \lambda , \mu ) = ( \mu  -   \lambda)^{2^n}
\Phi_n \left ( \frac{2 - \lambda^2 +  \mu^2}{ \mu - \lambda}  , 
\frac{-2}{\mu - \lambda}  \right ) . $$
This is exactly the formula from \cite{grizuk} which leads to the explicit computation of 
all eigenvalues.

\subsection{Numerical computation of spectra for BBS ($k \geq 2$)}
In order to analyze spectral characteristics of the transition operators for $k \geq 2$,
as a first step, we did numerical computation of the spectral distributions   for $k=2,3,4,5$.
Let us compare the histogram of the spectral distributions for $k=1$ and $2$.
Figures $2$ and $3$ present the histogram of the distributions of the multiple eigenvalues for $k=1$, $n=7$, and $k=2$, $n=14$
respectively.
Roughly we can see their structural similarity.

\begin{figure}
    \begin{minipage}[ht]{0.5\columnwidth}
        \begin{center}
            \includegraphics[height=4cm]{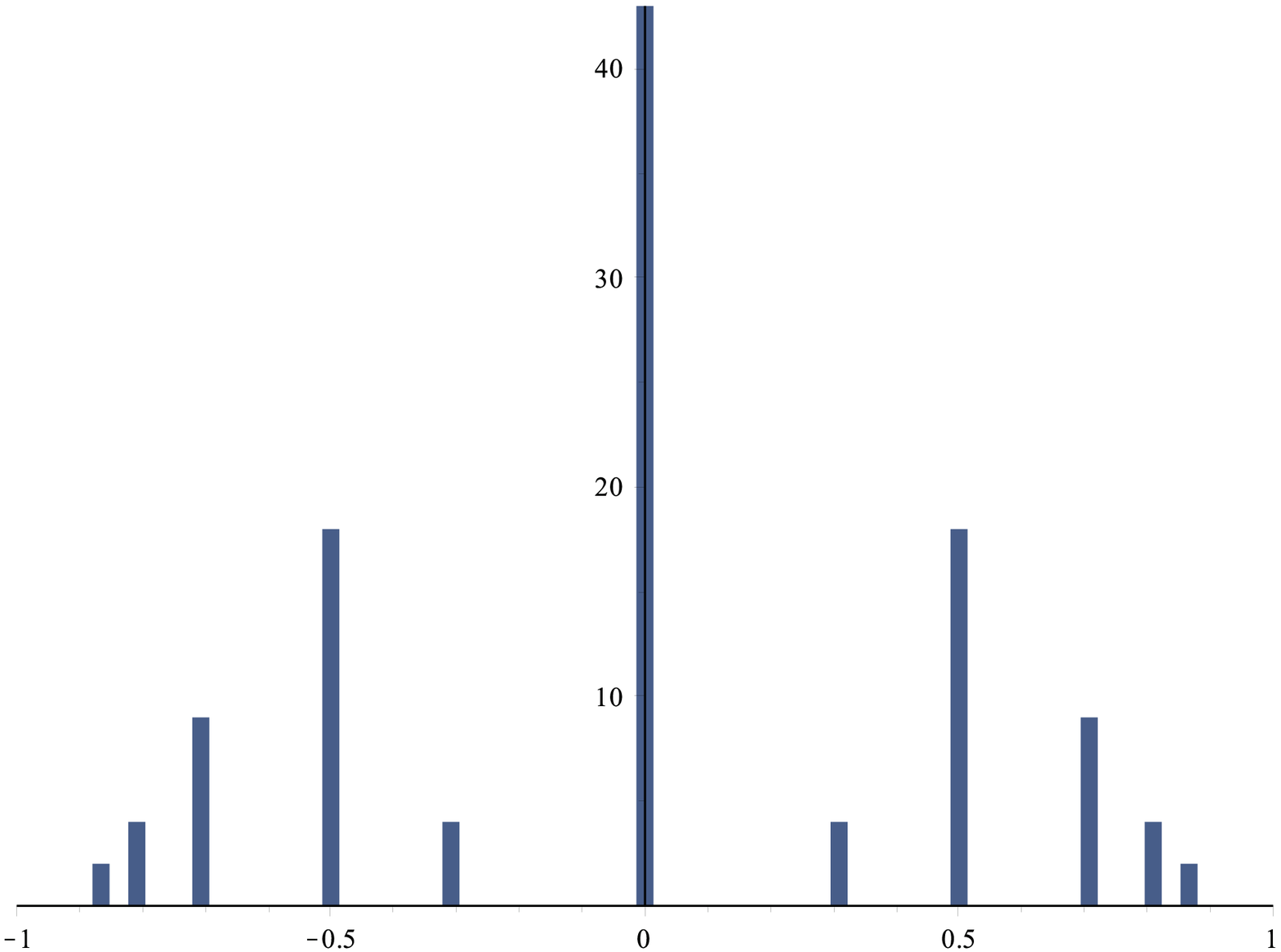}
        \end{center}
        \caption{Distribution of the multiple eigenvalues of $M_{k=1}^{(7)}$}
        \label{fig:left}
    \end{minipage}%
    \begin{minipage}[ht]{0.5\columnwidth}
        \begin{center}
            \includegraphics[height=4cm]{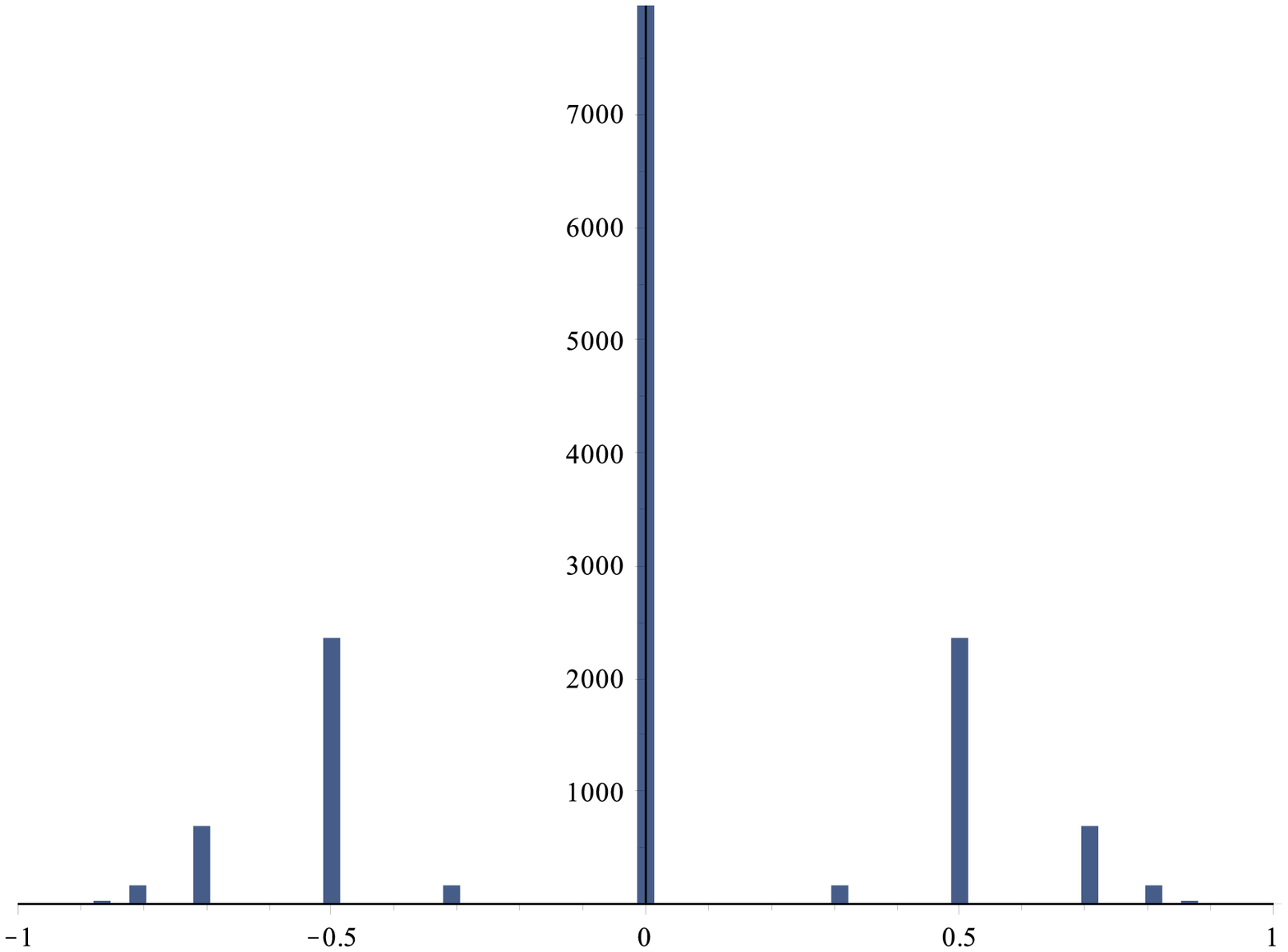}
        \end{center}
         \caption{Distribution of the multiple eigenvalues of $M_{k=2}^{(14)}$}
        \label{fig:right}
    \end{minipage}
\end{figure}

Let us see more detailed distributions for $k=1,2,\ldots,5$ by Tables \ref{table1}, \ldots, \ref{table4} below.
The tables present the distribution of the non-negative eigenvalues with multiplicities larger than or equal to $2$.
We have listed only non-negative eigenvalues, where negative ones appear symmetrically for $k=1$.
For $k=2,3,4,5$ cases also, negative eigenvalues appear almost symmetrically on their multiplicities,
except a few values. Actually their monotonicity with respect to $n$ hold.
We also present the growth of the rates of $r^{(n;k)}$ for $k=2$ in Figure \ref{fig:rates_r}.

Observe the following structural similarity of $k=2,3,4,5$ cases to  $k=1$:
\begin{enumerate}
 \item 
  The eigenvalues for $k \geq 2$,
  which are monotone increasing with respect to large $n$ coincide with 
  the ones for $k=1$.
  
\item One can find structural similarity of the distributions of the eigenvalues. 
Another multiple eigenvalues appear  on every $k$ steps for large $n$
  as is the case for $k=1$.
 The order of appearance of  the another multiple eigenvalues coincide.
 More concretely, for $k=1$, 
 another eigenvalue $\cos \frac{p}{n-1}\pi$ appear at the $n$-stage as multiple eigenvalues, 
 and then they grow monotonically.
 For $k=2$, there corresponds to $\cos \frac{2p}{n}\pi$ with $n =2,4, 6, \dots$.
 For general $k$, the eigenvalues are of the form $\cos\frac{p \pi}{\lfloor\frac{n-2}{k}\rfloor+1}$, 
where $\lfloor \ \rfloor$ is the largest integer not greater than itself (Gauss symbol).

\item 
Below in the tables 4 and 5, some of the eigenvalues are the extra ones which do not appear 
for $k=1$ case. They are included in the sets $\frac{\pm 1}{k+1}, \frac{\pm 2}{k+1}, \dots$.

\item  The rates  of the multiple eigenvalues 
$r^{(n;k)} = 2^{-n} \{ \ \sum_{i,j} m^{(n;k)}_{i,j} + \sum_j m^{(n;k)}(\frac{\pm j}{k+1}) \  \} $ in Figure \ref{fig:rates_r}
 seems to grow to $1$ with respect to $n$.
\end{enumerate}


Based on these observations, we would like to propose
the followings:
\begin{conj}\label{conj1}
Let $j$ be a non-negative integer and let $\hat{\text{Sp}}(M_{k=j}^{(n)}) \subset \text{Sp}(M_{k=j}^{(n)}) $ be the set of multiple eigenvalues.
Then:
$$ 
\text{Sp}(M_{k=1}^{\left(\lfloor (n-2)/j \rfloor\right)})=
\{1\}\cup\left(\hat{\text{Sp}}(M_{k=j}^{(n)})  \cap \hat{\text{Sp}}(M_{k=j}^{(n+1)})\right)
\quad \mbox{for} \quad n \ge 3.
$$
\end{conj}

\begin{conj}\label{conj2}
 Let $j$ be a non-negative integer. There are $n_{\lambda,j} \in {\mathbb N}$ so that the following equalities hold:
\begin{eqnarray*}
 \lim_{n\to\infty}\mbox{Sp}\left(M_{k=1}^{(n)}\right) =
 \lim_{n\to\infty}\left\{\lambda\in\mbox{Sp}\left(M_{k=j}^{(n)}\right)\,\big|\,
0<
		   m^{(n_{\lambda,j};j)}(\lambda)  \le \cdots  \le m^{(n-1;j)}(\lambda)
		   \le m^{(n;j)}(\lambda) \right\}.
\end{eqnarray*}
\end{conj}

\begin{table}[htb]
\begin{center}
\caption{Multiplicities of non-negative eigenvalues for $k=1$}
\label{table1}
\begin{tabular}{|c|cccccccccccc|}
\hline
 $n$ & $m_{1,2}^{(n;1)}$&$m_{1,3}^{(n;1)}$&$m_{1,4}^{(n;1)}$&$m_{1,5}^{(n;1)}$&$m_{2,5}^{(n;1)}$&$m_{1,6}^{(n;1)}$&$m_{1,7}^{(n;1)}$&$m_{2,7}^{(n;1)}$&$m_{3,7}^{(n;1)}$&$m_{1,8}^{(n;1)}$&$\cdots $& $m_{5,11}^{(n;1)}$\\
\hline
 1 & 1 &  0& 0& 0& 0& 0& 0& 0& 0& 0&  &0\\
 2 & 1 &  1& 0& 0& 0& 0& 0& 0& 0& 0&  &0\\
 3 & 3 &  1& 1& 0& 0& 0& 0& 0& 0& 0&  &0\\
 4 & 5 &  2& 1& 1& 1& 0& 0& 0& 0& 0&  &0\\
 5 & 11&  5& 2& 1& 1& 1& 0& 0& 0& 0&  &0\\
 6 & 21&  9& 4& 2& 2& 1& 1& 1& 1& 0&  &0\\
 7 & 43& 18& 9& 4& 4& 2& 1& 1& 1& 1&  &0\\
 8 & 85& 37&17& 8& 8& 4& 2& 2& 2& 1&  &0\\
 9 &171& 73&34&17&17& 8& 4& 4& 4& 2&  &0\\
10 &341&146&68&33&33&16& 8& 8& 8& 4&  &1\\
\hline
\end{tabular}
\end{center}
\end{table}

\begin{table}[htbp]
\begin{center}
\caption{Multiplicities of non-negative multiple eigenvalues for $k=2$}
\label{table2}
 \begin{tabular}{|c|ccccccccc|}
\hline
 $n$ & $m_{1,2}^{(n;2)}$&$m_{1,3}^{(n;2)}$&$ m_{1,4}^{(n;2)}$&$m_{1,5}^{(n;2)}$&$m_{2,5}^{(n;2)}$&$m_{1,6}^{(n;2)}$&$m_{1,7}^{(n;2)}$&$m_{2,7}^{(n;2)}$&$m_{3,7}^{(n;2)}$ \\
\hline
 1&  0  &   0& 0 & 0&0&0&0&0&0\\
 2&  1  &   0& 0 & 0&0&0&0&0&0\\
 3&  0  &   0& 0 & 0&0&0&0&0&0\\
 4&  4  &   0& 0 & 0&0&0&0&0&0\\
 5&  6  &   0& 0 & 0&0&0&0&0&0\\
 6&   22&   3& 0 & 0&0&0&0&0&0\\
 7&   42&   6& 0 & 0&0&0&0&0&0\\
 8&  104&  21& 3 & 0 &0&0&0&0&0\\
 9&  210&  50& 6 &0&0&0&0&0&0\\
10&  460& 118& 24&  3&  3&0&0&0&0\\
11&  930& 252& 54&  6&  6&0 &0&0&0\\
12& 1940& 551&144& 25& 25& 3&0&0&0\\
13& 3906&1134&306& 60& 60& 6&0&0&0\\
14& 7966&2359&692&165&165&28& 3 &3 &3\\
15&16002&4788&1434&366&366&66& 6 &6 &6\\
\hline
\end{tabular}
\end{center}
\end{table}

\begin{table}[htbp]
\begin{center}
\caption{Multiplicities of non-negative multiple eigenvalues for $k=3$}
\label{table3}
\begin{tabular}{|c|cccccc|}
\hline
 $n $&$ m^{(n;3)}(\frac{1}{4})$&$m_{1,2}^{(n;3)}$&$m_{1,3}^{(n;3)}$&$m_{1,4}^{(n;3)}$&$m_{1,5}^{(n;3)}$&$m_{2,5}^{(n;3)}$ \\
\hline
 1 &0&   0&   0&   0&   0&   0\\
 2 &0&   0&   1&   0&   0&   0\\
 3 &0&   1&   0&   0&   0&   0\\
 4 &2&   0&   0&   0&   0&   0\\
 5 &0&   4&   0&   0&   0&   0\\
 6 &0&   7&   0&   0&   0&   0\\
 7 &0&  26&   0&   0&   0&   0\\
 8 &0&  56&   2&   0&   0&   0\\
 9 &0& 151&   7&   0&   0&   0\\
10 &0& 332&  26&   0&   0&   0\\
11 &0& 776&  68&   2&   0&   0\\
12 &0&1653& 196&   7&   0&   0\\
13 &0&3640& 464&  30&   0&   0\\
14 &0&7604&1152&  80& 2 &2\\
15 &0&16157&2570&256& 7 &7\\
\hline
\end{tabular}
\end{center}
\end{table}

\begin{table}[htbp]
\begin{center}
\caption{Multiplicities of non-negative multiple eigenvalues for $k=4$ and $5$}
\label{table4}
\begin{tabular}{|c|cccc|}
\hline
 \multicolumn{5}{|l|}{$k=4$} \\
\hline
 $n $&$ m^{(n;4)}(\frac{1}{5})$&$m_{1,2}^{(n;4)}$&$m_{1,3}^{(n;4)}$&$m_{1,4}^{(n;4)}$ \\
\hline
 1 &0&   0& 0&   0\\
 2 &0&   0& 0&   0\\
 3 &0&   0& 1&   0\\
 4 &0&   1& 0&   0\\
 5 &2&   0& 0&   0\\
 6 &0&   3& 0&   0\\
 7 &1&   6& 0&   0\\
 8 &0&  29& 0&   0\\
 9 &3&  62& 0&   0\\
10 &0& 185& 2&   0\\
11 &5& 418& 6&   0\\
12 &0&1061&31&   0\\
13 &9&2332&80&   0\\
14 &0&5427&265&  2\\
15 &15&11704&652&6\\
\hline
\end{tabular}
\begin{tabular}{|c|cccc|}
\hline
 \multicolumn{5}{|l|}{$k=5$} \\
\hline
 $n $&$ m^{(n;5)}(\frac{1}{6}) $&$ m^{(n;5)}(\frac{2}{6})$&$ m_{1,2}^{(n;5)}$&$m_{1,3}^{(n;5)}$ \\
\hline
 1 &0&0&   0& 0\\
 2 &0&0&   0& 0\\
 3 &0&0&   0& 0\\
 4 &0&0&   0& 1\\
 5 &0&1&   1& 0\\
 6 &1&0&   0& 0\\
 7 &0&1&   4& 0\\
 8 &0&0&   6& 0\\
 9 &0&3&  33& 0\\
10 &0&0&  69& 0\\
11 &0&5& 220& 0\\
12 &0&0& 500& 2\\
13 &0&9&1333& 6\\
14 &2&0&3002&34\\
15 &0&15&7327&93\\
\hline
\end{tabular}
\end{center}
\end{table}

\begin{figure}
    \begin{center}
     \includegraphics[height=7cm]{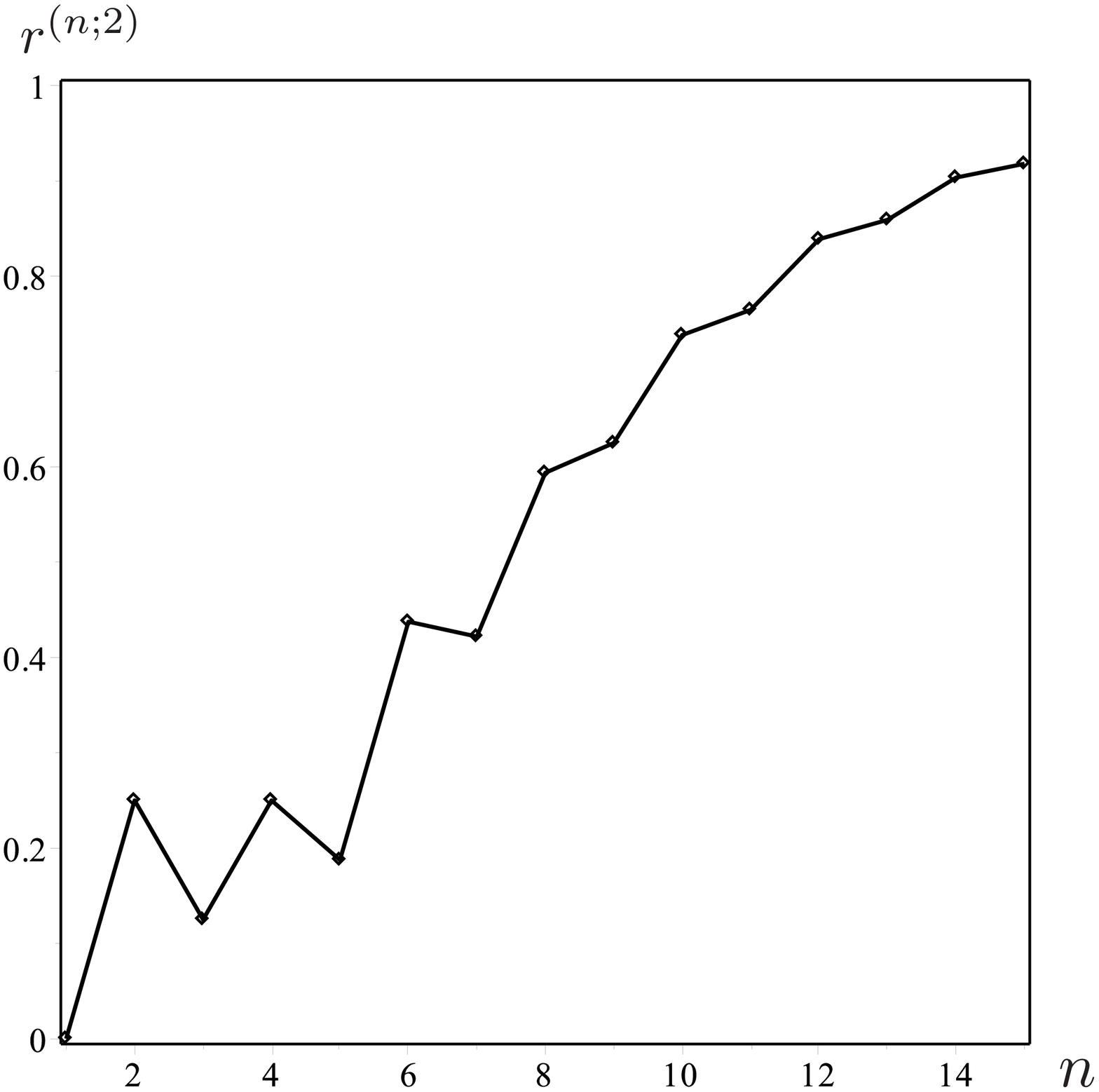}
    \end{center}
    \caption{Rates of the multiple eigenvalues $r^{(n;k)}$ for $k=2$}
    \label{fig:rates_r}
\end{figure}

So far we have found some  similarity of spectral distributions for various $k \geq 1$.
It is quite unexpected  for us to find any kind of structural similarity
among BBS$_k$ and the lamplighter automaton, since BBS$_{k=1}$ is dynamically
 translation invariant, while BBS$_{k \geq 2}$ behave  essentially nonlinear.
It would be reasonable to expect to see more concrete dynamical similarity for several $k \geq 1$.
On the other hand extra appearance of new eigenvalues are observed for $k \geq 2$,
which might lead to see essential difference of dynamics among BBS$_k$
(see $(3)$ above).
Combination  with these opposite phenomena will lead us with much deeper understanding
 of BBS.

\section{Ergodicity of the transition operators for BBS translation}

\subsection{Ergodicity on the boundary of the binary tree}
Let $\{M_1^{(n)}\}_{n =1,2, \dots}$ be the family of transition operators 
for lamplighter or BBS$_{k=1}$ automata.
We have verified that those are stochastic $2^n$ by $2^n$ matrices 
equipped with the canonical maps:
$$\dots \to M_1^{(n+1)} \to M_1^{(n)} \to \dots$$

\begin{df}
Let $M$ be a stochastic  $k$ by $k$ matrix. 
$M$ is ergodic, if there is $s_0 \geq 1$ and $\alpha >0$ 
so that inequalities:
$$m_{i,j}^{(s_0)} \geq \alpha$$
hold for all $i,j$, where 
 $M^{s} =(m_{i,j}^{(s)})_{1 \leq i,j \leq k}$.
 \end{df}
For stochastic matrix, if the above property is satisfied for some $s_0$, 
then the same property holds for all $s \geq s_0$.

Recall the fundamental result on ergodicity:
 \begin{thm}\label{thm:unique probability distribution}
Let $M$ be a stochastic $k$ by $k$ matrix, and consider the 
associated transition chain on the space  $X= \{1, \dots, k\}$. 
If $M$ is ergodic, then there is a unique probability distribution $\pi$ on $X$ 
which satisfies two   properties
{\normalfont (1)} $\pi M =\pi$, and {\normalfont (2)} $\lim_{s \to \infty} m_{i,j}^{(s)} = \pi_j$.
\end{thm}
The unique probability distribution $\pi = (\pi_1, \dots, \pi_k)$ is
called the stationary distribution with respect to $M$.

\begin{lem}\label{lem:ergoticity and eigenvalue}
$M$ is ergodic, if and only if the spectrum of $M$ satisfies 

{\normalfont (1)} the multiplicity of the eigenvalue $1$ is just $1$, and 

{\normalfont (2)} it does not contain $-1$.
\end{lem}
\begin{proof}
Suppose $M$ is ergodic.
Let $v_1$ and $v_2$ be two orthogonal eigenvectors with eigenvalue $1$.
Then $v_i M =v_i$ hold, and so:
$$\langle v_1 M^{2s}, v_2\rangle = \langle v_1 M^{s} , v_2 M^{s} \rangle =   \langle v_1,v_2\rangle=0$$
must hold. 
Let $a_i  $ be the sum of coordinates of $v_i$.
Then $a_i$ can not be zero, since 
 $v_i = \lim_{s \to \infty} v_i M^{s} =a_i \pi$ hold by Theorem \ref{thm:unique probability distribution}.
By letting $s \to \infty$ in the above equalities, 
it follows $\pi =0$ is zero vector, which is a contradiction,
since $\pi_i \geq \alpha >0$.
So the multiplicity of the eigenvalue $1$ must be less than or equal to $1$.
It is at least $1$ because constant vectors have eigenvalue $1$.

As we noticed  that the limit exists:
$$wM^s \equiv (x_1, \dots,x_n)M^{s} \to (a \pi_1, \dots, a \pi_n)$$
by Theorem \ref{thm:unique probability distribution}, 
where $a = \sum_{i=1}^n x_i$. 
But if  $w$ is an  eigenvector with eigenvalue $-1$, then
$w M^{s} \in \{ w, -w\}$ oscillates, which is a contradiction.

Suppose the above two properties hold.
Let $\{ v_1, \dots, v_k\} $ be the orthogonal eigenvectors such that $v_1$ corresponds to the  eigenvalue $1$.
Then for any $v =\sum_{i=1}^k a_i v_i$,
$$\lim_{s \to \infty} v M^{s} = a_1 v_1 +  \lim_{s\to \infty}  \sum_{i=2}^k \lambda_i^{s} a_i v_i =a_1v_1$$
hold, since $-1 < \lambda_i <1$ hold for $i \geq 2$.

Suppose $M$ is not ergodic, i.e. for every $s$, there exist $i,j$ such that 
$m_{i,j}^{(s)} = 0$  hold. Let $\delta_i=(0, \dots, 0,1,0, \dots)$.
Then $\langle\delta_i M^{s} , \delta_j\rangle=m_{i,j}^{(s)}=0$ holds.
It follows that there exist $i,j$ such that $ \langle \delta_i M^{l} , \delta_j \rangle=0$ hold for infinitely many $l$.
So it also holds for $l \to \infty$.
It follows that  $\delta_i$ or $\delta_j$ is orthogonal to $v_1$.
Since $M$ is stochastic,  we can put $v_1= (1, \dots,1)$, and so this is a contradiction.
This completes the proof.
\end{proof}
\begin{rem}
For the stochastic matrix, 
the property {\normalfont (1)} is equivalent to connectivity,
 and property {\normalfont (2)} is to non bi-partiteness of the associated graph.
\end{rem}

 \begin{cor}\label{cor:5.4}
 Let $M_L^{(n)}$ and $M_B^{(n)}$
 be the transition operators for  the lamplighter and  BBS$_{k=1}$ automata, respectively.
 Then they are all ergodic.
 \end{cor}
 \begin{proof}
 The result follows from our computation of their spectra in Theorem $7.1$ with Lemma \ref{lem:ergoticity and eigenvalue}.
 \end{proof}

\subsection{On automorphisms of the  tree}
Let $T$ be the binary tree, and $T_n$ be $n$-th level set.
Then the transition operators satisfy:
$$M^{(n+1)}|_{T_n} =M^{(n)}.$$

Let us consider the canonical maps:
$$\dots \to M^{(n+1)} \to M^{(n)} \to \dots$$
and take the projective limit:
$$M \equiv \lim_{\leftarrow n} M^{(n)}.  $$
 $M$ gives an ergodic transition chain on $\partial T$, if $M^{(n)}$ are ergodic.

 \begin{prop}\label{lem8.5}
 There are no automorphisms of $T$ which conjugate 
 between $M_{L}^{(n)} $ and $ M_{B}^{(n)} $.
 \end{prop}
\begin{proof}
If there were an automorphism of the tree which would conjugate two 
operators on some level $n$ it would also conjugate these operators on the 
previous levels. Thus it is enough to prove the statement for the level $n=2$.
For this level the operator corresponding to the BBS system has $(2,0,0,2)$ on the diagonal
and the operator corresponding to the lamplighter has $(0,0,2,2)$ on the diagonal.
The last one under the tree automorphism can be transformed to itself or
$(2,2,0,0)$ only.
\end{proof}

\vspace{3mm}

\section{Conjugacy by permutation for BBS translation}
Let $M_B^{(n)} , M_L^{(n)}  \in \mbox{Mat}(2^n \times 2^n; \Z)$.
Let us denote the set of indices as $I_n=\{0,1,\ldots,2^n-1\}$.
We denote the concatenation of two vectors $u\in \C^{n}$ and
$v\in\C^{m}$ by $(u,v)\in \C^{n+m}$.
For $c \in I_n$, 
consider the binary expansion :
 $$c= \sum_{j=1}^{n} c_{j} \,2^{n-j} \in I_n$$
where $c_j \in  {\mathbb Z}_2$,
 which we denote as:
$$
c= (c_{1},c_{2}\cdots,c_{n})_2  .$$

In this section, we verify the following:
\begin{thm}\label{thm:BBS to LL}
 There exists a family of the transformation matrices $\sigma_n$ such that
\begin{align*}
 \sigma_n^* M_B^{(n)} \sigma_n =M_L^{(n)}
\end{align*}
hold, 
where $\sigma_n$ is determined by the permutation vector 
$\displaystyle \widehat\sigma_n=
\left(\mu_0^{(n)}, \mu_1^{(n)} , \cdots , \mu_{2^n-1}^{(n)}\right)$ by
\begin{align*}
 \sigma_n e_{j} = e_{\mu_j^{(n)}}
\end{align*}
for any $j \in I_n$, where $e_j=(0, \dots, 0,1, 0, \dots, 0)^*$. 

The permutation vectors $\widehat\sigma_n=(\mu_{0}^{(n)},\mu_{1}^{(n)},\ldots,\mu_{2^{n}-1}^{(n)}) $ are uniquely determined by
\begin{itemize} 
  \item $\widehat\sigma_1=(\mu_0^{(1)},\mu_1^{(1)})=(0,1)$,
  \item there exists a binary sequence
	$\nu^{(n)}=(\nu^{(n)}_0,\nu^{(n)}_1,\ldots,\nu^{(n)}_{2^{n-1}-1})
	\in \{0,1\}^{2^{n-1}}$ such that
\begin{align}
\widehat\sigma_n 
=
(\widehat\sigma_{n-1},\widehat\sigma_{n-1})+ 2^{n-1}(1- \nu_0^{(n)},\nu_0^{(n)}, 
\ldots, 1-\nu_{2^{n-1}-1}^{(n)},\nu_{2^{n-1}-1}^{(n)}),
\label{def:sigma}
\end{align}
\item the binary sequences $\nu^{(n)} \in \{0,1\}^{2^n-1}$ 
      are determined by use of the  binary pattern $g$ of the Sierpinski gasket
      as follows:
\begin{eqnarray}
 \nu^{(n)}=(T_{g^{(n)}_1}\circ T_{g^{(n)}_{2}}\circ \cdots \circ T_{g^{(n)}_{n-1}})(0) \label{def:nu}
\end{eqnarray}
for $n \ge 2$. The operator $T_{\alpha}$ is defined by
\begin{eqnarray*}
 T_{\alpha}(s_1,s_2,\ldots,s_m) = \left\{\begin{array}{ll}
		 (s_1,s_2,\ldots,s_m,s_1,s_2,\ldots,s_m) & \mbox{if $\alpha=0$}\\
		 (s_1,s_2,\ldots,s_m,1-s_1,1-s_2,\ldots,1-s_m) & \mbox{if $\alpha =1$}\\
			\end{array}\right..
\end{eqnarray*}
Here the binary pattern $g$ of  the Sierpinski gasket
 is given  by
 $g_{1}^{(n)}=g_{n}^{(n)}=1$ and 
 $$g^{(n)}_{m}=
      g^{(n-1)}_{m-1}+g^{(n-1)}_{m} \mod 2 \,\,$$ for
      $m=2,3,\ldots,n-1$ and $n=1,2,\ldots$ .
 \end{itemize}
\end{thm}
\begin{rem}
{\rm (1)}
Let us see the orbit of $g$:
\begin{align*}
 &g^{(1)}=(g^{(1)}_1)=(1), \\
 &g^{(2)}=(g^{(2)}_1,g^{(2)}_2)=(1,1),\\
 &g^{(3)}=(g^{(3)}_1,g^{(3)}_2,g^{(3)}_3)=(1,0,1),\\
 &g^{(4)}=(g^{(4)}_1,g^{(4)}_2,g^{(4)}_3,g^{(4)}_4)=(1,1,1,1),\\
 &g^{(5)}=(g^{(5)}_1,g^{(5)}_2,g^{(5)}_3,g^{(5)}_4,g^{(5)}_5)=(1,0,0,0,1),\\
 &g^{(6)}=(g^{(6)}_1,g^{(6)}_2,g^{(6)}_3,g^{(6)}_4,g^{(6)}_5,g^{(6)}_6)=(1,1,0,0,1,1),
\\ 
&g^{(7)}=(g^{(7)}_1,g^{(7)}_2,g^{(7)}_3,g^{(7)}_4,g^{(7)}_5,g^{(7)}_6,g^{(7)}_7)=(1,0,1,0,1,0,1),
\ldots.
\end{align*}      
which gives the pattern of the Sierpinski gasket.

{\rm (2)} Another formula of $g$ is given by:
$$g_m^{(n)} = \frac{(n-1)!}{(m-1)!(n-m)!} \mod 2$$
\end{rem}

\begin{cor} The formulas hold for all $k \in I_{n-1}$:
$$\mu_{2k}^{(n)} + \mu_{2k+1}^{(n)} =2^n-1,   \quad \mu_{2k}^{(n)} \in 2 I_{n-1}.$$
\end{cor} 
\begin{proof}
We proceed by induction. Suppose the conclusion holds up to $n-1$.
It follows from (\ref{def:sigma}) that 
\begin{align*}
 \mu_{2k}^{(n)} + \mu_{2k+1}^{(n)}  
& = \mu_{2k}^{(n-1)} + \mu_{2k+1}^{(n-1)} +2^{n-1} \\
& = 2^n-1.
\end{align*}

The latter formula follows immediately.
\end{proof}

Let us denote $[n] \in \Z_{2}$ by the image of $n$ by $\Z \to \Z_2$.

 \begin{lem}
{\rm (i)}
 $\nu^{(n)}$ is given by
\begin{align*}
\label{nu to g}
 \nu^{(n)}_k=\nu^{(n)}_{(k_1,k_2,\cdots, k_{n-1})_2}=\bigg[\sum_{j=1}^{n-1}k_j g_j^{(n)}\bigg].
\end{align*}

{\rm (ii)}
\label{prop permutation matrix}
 $\widehat\sigma_n=(\mu_0^{(n)},\mu_1^{(n)},\ldots,\mu_{2^n-1}^{(n)})$ is a permutation vector of $I_n$, that is,  
 $$\mu_j^{(n)} \in I_n, \quad \mu_j^{(n)} \ne \mu_{j'}^{(n)} \text{ for all distinct pairs } j \ne j'$$
\end{lem}
\begin{proof}
(i) 
Let us rewrite $T_{\alpha}$ as:
$$
 T_{\alpha}(s_1,s_2,\ldots,s_m) = ([s_1],[s_2],\ldots,[s_m],[s_1+\alpha],[s_2+\alpha],\ldots,[s_m+\alpha]).$$

For example we see the case of $\nu^{(4)}$,
\begin{align*}
\nu^{(4)}=& (T_{g_1^{(4)}}\circ T_{g_2^{(4)}}\circ T_{g_3^{(4)}})(0)
\\
=&
(0,[g_3^{(4)}],[g_2^{(4)}],[g_2^{(4)}+g_3^{(4)}],[g_1^{(4)}],[g_1^{(4)}+g_3^{(4)}],[g_1^{(4)}+g_2^{(4)}],[g_1^{(4)}+g_2^{(4)}+g_3^{(4)}]),
\end{align*}
and 
\begin{align*}
& \nu^{(4)}_{(0,0,0)_2}=[0+0+0], \ 
\nu^{(4)}_{(0,0,1)_2}=[0+0+g_3^{(4)}], \
\nu^{(4)}_{(0,1,0)_2}=[0+g_2^{(4)}+0], \\
& \nu^{(4)}_{(0,1,1)_2}=[ 0+g_2^{(4)}+g_3^{(4)}], \\
& \nu^{(4)}_{(1,0,0)_2}=[g_1^{(4)}+0+0],
\nu^{(4)}_{(1,0,1)_2}=[g_1^{(4)}+0+g_3^{(4)}],\ldots.
\end{align*}

Let us consider the general case.
For any $h_j \in \{0,1\}$,
let us define
\begin{align*}
& v=(v_0, \cdots, v_{2^{n-1}-1}) =(T_{h_{n-1}} \circ T_{h_{n-2}} \circ \cdots \circ T_{h_1})(0) \\
& \tilde{v}=(\tilde{v}_0, \cdots, \tilde{v}_{2^n-1}) =(T_{h_n} \circ T_{h_{n-1}} \circ T_{h_{n-2}} \circ \cdots \circ T_{h_1})(0)  \\
&  \quad = (v,v+h_n) \mod 2.
\end{align*}
If $v_k = v_{(k_1, \cdots, k_{n-1})_2}= [\sum_{j=1}^{n-1} k_j h_{n-j}]$, 
then 
$$\tilde{v}_{\tilde{k}} =\tilde{v}_{(\tilde{k}_1,\ldots,\tilde{k}_n)_2} =[v_k +\tilde{k}_1h_n]= \bigg[\sum_{j=1}^n \tilde{k}_j h_{n-j+1}\bigg] $$
for $\tilde{k} \in \{0, \cdots, 2^n-1\}$,
since $k_1=\tilde{k}_2, \cdots, k_{n-1}= \tilde{k}_n$ hold.

If we insert $g_i^{(n)}$ into $h_{n-i}$ in $v$, then we obtain
 $\nu^{(n)} $, that is 
 $$\nu^{(n)}_k =  \bigg[\sum_{j=1}^{n-1} k_j g_j^{(n)}\bigg].$$

\vspace{3mm}

(ii)
We proceed by induction.
For $n=1$, $\widehat\sigma_1=(0,1)$ corresponds to the identity
over $I_1=\{0,1\}$.

 Suppose that  the conclusion holds up to $n-1$ so that 
  $\widehat\sigma_{n-1}$ be a permutation vector of $I_{n-1}$.
 It follows from the expression (i) 
 that 
 for any $k_2,\ldots,k_{n-1} \in \{0,1\}$, the equalities hold:
\begin{align*}
 [\nu^{(n)}_{(0,k_2,\cdots, k_{n-1})_2}
+ \nu^{(n)}_{(1,k_2,\cdots, k_{n-1})_2}]=\bigg[g_1^{(n)}+ 2\sum_{j=2}^{n-1}k_j g_j^{(n)}\bigg] = g_1^{(n)}=1.
\end{align*}
In particular 
$\nu^{(n)}_{(0,k_2,\cdots, k_{n-1})_2} \ne \nu^{(n)}_{(1,k_2,\cdots, k_{n-1})_2} $, and hence
\begin{align*}
| \mu_{j}^{(n)} - \mu_{j+2^{n-1}}^{(n)}| = g_1^{(n)} 2^{n-1} =2^{n-1}
\end{align*}
hold for any $j \in I_{n-1}$.

By the assumption, 
 $\widehat\sigma_{n-1}$ is a permutation vector on $I_{n-1}$ so that 
$\mu_{j}^{(n)} \ne \mu_{l}^{(n)}$ hold for any distinct pair $0 \leq j,l \leq 2^{n-1}-1$.
Since the value of $\mu_{j}^{(n)}$ does not exceed $2^n$,
it follows that $\mu_{j}^{(n)} \ne \mu_{l}^{(n)}$ hold for any distinct pair $0 \leq j,l \leq 2^n-1$, and hence
  $\widehat\sigma_n$ must be a permutation vector of $I_n$.
\end{proof}

\begin{prop}\label{prop sigma_n^2}
$\widehat\sigma_n^2=$ id hold  on $I_n=\{0,1, \dots, 2^n-1\}$. 
\end{prop}
\begin{proof}
Let $k \in I_n$.
For $k=(k_1,k_2,\ldots,k_n)_2$, let us denote the corresponding  binary expansions:
\begin{align*}
 \widehat\sigma_n(k)&=\widehat\sigma_n((k_1,k_2,\ldots,k_n)_2)=(k_1',k_2',\ldots,k_n')_2,\\
 \widehat\sigma_n^2(k)&=\widehat\sigma_n((k_1',k_2',\ldots,k_n')_2)=(k_1'',k_2'',\ldots,k_n'')_2,
\end{align*}
respectively.
Firstly let us verify the formulas:
\begin{eqnarray}
\label{sigma n k}
\widehat\sigma_n(k)
 &= 
\left(\widehat\nu_{(k_1,k_2,\ldots,k_{n-1})_2}^{(n;k_n)}, \widehat\nu_{(k_2,\ldots,k_{n-1})_2}^{(n-1;k_n)}, \ldots,
\widehat\nu_{(k_{n-1})_2}^{(2;k_n)},k_n\right)_2
\end{eqnarray}
where $\widehat\nu^{(n-\kappa+1;k_n)}_{(k_{\kappa},\ldots,k_{n-1})_2} \in \{0,1\}$  is defined by
\begin{align*}
 \widehat\nu_{(k_{\kappa},\ldots,k_{n-1})_2}^{(n-{\kappa}+1;k_n)} 
\equiv  [1 + \nu_{(k_{\kappa},\ldots,k_{n-1})_2}^{(n-{\kappa}+1)}+k_n] =
\bigg[1+ \sum_{i=1}^{n-{\kappa}+1}k_{{\kappa}+i-1} g_i^{(n-{\kappa}+1)}\bigg].
\end{align*}
Since $\hat{\sigma}_1=(0,1)$, one can see that $\widehat\sigma_1(0)=0$ and $\widehat\sigma_1(1)=1$.

Notice  that  $\nu^{(2)} = T_{g_1^{(2)}}(0) = T_1(0) =(0,1)$ and hence
\begin{align*}
 \hat{\sigma}_2 & = 
 (\hat{\sigma}_1, \hat{\sigma}_1) + 2(1- \nu_0^{(2)}, \nu_0^{(2)} , 1- \nu_1^{(2)}, \nu_1^{(2)}) \\
 & =(0,1,0,1) +2(1-0, 0, 1-1,1) =(2,1,0,3).
 \end{align*}
 It  can be presented as
\begin{align*}
 & (\widehat{\nu}_{k_1}^{(2;k_2)} , k_2)_2 
 =(1+k_1g_1^{(2)}+k_2,k_2)_2 = 2[1+k_1+k_2]+k_2,
\end{align*}
for any $k \in I_2$. In fact the equalities follow from 
direct computations:
  \begin{align*}
&   \widehat\sigma_2(0)=2[1+0+0]+0=2, \quad
  \widehat\sigma_2(1)=2[1+0+1]+1=1, \\
  & \widehat\sigma_2(2)=2[1+1+0]+0=0, \quad
  \widehat\sigma_2(3)=2[1+1+1]+1=3.
  \end{align*}

  Suppose the formula holds up to $n-1$. Then 
  we have the equalities:
\begin{align}
\widehat\sigma_n(k)&=\widehat\sigma_n((k_1,k_2,\ldots,k_n)_2)
\\\nonumber &
=\widehat\sigma_{n-1}((k_2,k_3,\ldots,k_n)_2)
+ \left\{ \begin{array}{lr}
   2^{n-1}\nu_{(k_1,k_2,\ldots,k_{n-1})_2}^{(n)}&  \text{ if  $k_n=1$}\\
   2^{n-1}(1-\nu_{(k_1,k_2,\ldots,k_{n-1})_2}^{(n)})&  \text{ if  $k_n=0$}
	  \end{array} \right.
&
\\\nonumber
&=(0,\widehat\nu_{(k_2,k_3,\ldots,k_{n-1})_2}^{(n-1;k_n)},\cdots,\widehat\nu_{(k_{n-1})_2}^{(2;k_n)},k_n)_2
+  (1+\nu_{(k_1,k_2,\ldots,k_{n-1})_2}^{(n)}+k_n,0,\ldots,0)_2
&
\\\nonumber &=
\left(\widehat\nu_{(k_1,k_2,\ldots,k_{n-1})_2}^{(n;k_n)}, \widehat\nu_{(k_2,\ldots,k_{n-1})_2}^{(n-1;k_n)}, \ldots,\widehat\nu_{(k_{n-1})_2}^{(2;k_n)},k_n\right)_2.
\end{align}
So it holds for $n$.

Next by use of  (\ref{sigma n k}),
we obtain the equalities:
\begin{align}
\label{sigma b to g}
 k_{\kappa}' &= \bigg[1+ \sum_{i=1}^{n-{\kappa}+1}k_{{\kappa}+i-1} g_i^{(n-{\kappa}+1)}\bigg], \quad k_{n}'=k_n,
\\
 k_{\kappa}''&= \bigg[1+ \sum_{i=1}^{n-{\kappa}+1}k_{{\kappa}+i-1}' g_i^{(n-{\kappa}+1)} \bigg], \quad k_{n}''=k_n,
\end{align}
where $\kappa=1,2,\ldots,n-1$.
For $n-\kappa \in \{1,2,\ldots,n-1 \}$,
\begin{align*}
 k_{n-\kappa}''&=\bigg[1+ \sum_{j=1}^{\kappa+1} k_{n-\kappa+j-1}'g_j^{(\kappa+1)}\bigg]
\\&=\bigg[1+ \sum_{j=1}^{\kappa} (1+ \sum_{l=1}^{\kappa-j+2} k_{n-\kappa+j+l-2}g_l^{(\kappa-j+2)})g_j^{(\kappa+1)} + k_{n}' g_{\kappa+1}^{(\kappa+1)}\bigg] \\
& = \bigg[1+ \sum_{j=1}^{\kappa} (1+ \sum_{l=1}^{\kappa-j+2} k_{n-\kappa+j+l-2}g_l^{(\kappa-j+2)})g_j^{(\kappa+1)} + k_{n} g_1^{(1)} g_{\kappa+1}^{(\kappa+1)}] \\
& =\bigg[1+\sum_{j=1}^{\kappa}\left(g_j^{(\kappa+1)}\right)+
k_{n-\kappa}g_{1}^{(\kappa+1)}g_{1}^{(\kappa+1)}   \\
& \qquad 
+ (\sum_{j=1}^{\kappa}  \sum_{l=1}^{\kappa-j+2} k_{n-\kappa+j+l-2}g_l^{(\kappa-j+2)} g_j^{(\kappa+1)} )
 - k_{n-\kappa}g_{1}^{(\kappa+1)}g_{1}^{(\kappa+1)} + k_{n} g_1^{(1)} g_{\kappa+1}^{(\kappa+1)}\bigg]
 \\
& =\bigg[1+\sum_{j=1}^{\kappa}\left(g_j^{(\kappa+1)}\right)+k_{n-\kappa}g_{1}^{(\kappa+1)}g_{1}^{(\kappa+1)}  + \sum_{j=2}^{\kappa+1} k_{n-\kappa+j-1} \left(\sum_{l=1}^{j} g_{j-l+1}^{(\kappa-l+2)}g_{l}^{(\kappa+1)}\right)\bigg] 
\\& = k_{n-\kappa},
\end{align*}
where we have used the equalities
$g_1^{(1)}=g_1^{(\kappa+1)}=g_{\kappa+1}^{(\kappa+1)}=1,
 [1+\sum_{j=1}^{\kappa}\left(g_j^{(\kappa+1)}\right)]=0$,
\begin{align*}
 g_{j-l+1}^{(\kappa-l+2)}g_{l}^{(\kappa+1)} &
=\bigg[\dfrac{\kappa!}{(\kappa-j+1)!(j-l)!(l-1)!}\bigg]
=g_{j-(j-l+1)+1}^{(\kappa-(j-l+1)+2)}g_{j-l+1}^{(\kappa+1)},
\\
g_{m+1}^{(\kappa-m+1)}g_{m+1}^{(\kappa+1)}&=
g_{2m+1-(m+1)+1}^{(\kappa-(m+1)+2)}g_{m+1}^{(\kappa+1)}
=\bigg[\dfrac{\kappa!}{(\kappa-2m)!\,m!\,m!}\bigg]
\\&
=\bigg[\dfrac{(2m)!}{m!\,m!}\dfrac{\kappa!}{(\kappa-2m)!\,(2m)!}\bigg]=
\bigg[ 2 \binom{2m-1}{m} {\kappa \choose 2m} \bigg]
=0, 
\end{align*}
and 
\begin{align*}
\bigg[\sum_{l=1}^{j} g_{j-l+1}^{(\kappa-l+2)}g_{l}^{(\kappa+1)}\bigg]
& =
\begin{cases}
\bigg[2 \sum_{l=1}^{m} g_{j-l+1}^{(\kappa-l+2)}g_{l}^{(\kappa+1)}\bigg]& (j=2m)\\[2.5mm]
\bigg[2 \sum_{l=1}^{m} g_{j-l+1}^{(\kappa-l+2)}g_{l}^{(\kappa+1)} + g_{m+1}^{(\kappa-m+1)}g_{m+1}^{(\kappa+1)}\bigg]  & (j=2m+1)\\
       \end{cases}  \\
       &= 0.
\end{align*}
Hence  $\widehat\sigma_n^2(k)=(k_1,k_2,\ldots,k_n)_2=k$ holds. 
\end{proof}

\begin{proof}[Proof of Theorem \ref{thm:BBS to LL}]
Let $j, k\in I_n$,
and  denote the binary expansions:
\begin{align*}
  &j=(j_1,j_2,\ldots,j_n)_2, \quad k=(k_1,k_2,\ldots,k_n)_2, \\
  &\widehat\sigma_n(j)=(j_1',j_2',\ldots,j_n')_2, \quad \widehat\sigma_n(k)=(k_1',k_2',\ldots,k_n')_2.
\end{align*}

Let us denote the two generating elements $a_0, a_1$
of the dynamics for 
 the lamplighter operators  (\ref{lamplighter operators}) and the BBS  (\ref{bbs operators})
 by: 
 $$a_L^{(0)}, \ a_L^{(1)}, \ a_B^{(0)}, \ a_B^{(1)}$$ respectively.
These operators satisfy the following recursive relations for $\varepsilon = 0,1$:
$$ a_L^{(\varepsilon)} = ( [j+k+\varepsilon]a_L^{(k)})_{0 \le j,k \le 1}, \quad
 a_B^{(\varepsilon)} = ( [j+1+\varepsilon]a_B^{(k)})_{0 \le j,k \le 1}.
$$
By applying these formulas repeatedly, 
we obtain four matrices of the size 
$2^n$ by $2^n$  given by 
\begin{align*}
& a_B^{(\varepsilon;n)}=\bigg( \alpha^{(\varepsilon)}_{j,k} \bigg)_{0 \le j,k <2^n} =\bigg( [(j_1+1+\varepsilon)(j_2+1+k_1)\cdots (j_{n}+1+k_{n-1})] \bigg)_{0 \le j,k \le 2^n-1}, \\
& a_L^{(\varepsilon;n)}=\bigg( \beta^{(\varepsilon)}_{j,k} \bigg)_{0 \le j,k <2^n}=\bigg( [(j_1+k_1+\varepsilon)(j_2+k_2+k_1)\cdots (j_{n}+k_n+k_{n-1}) ]\bigg)_{0 \le j,k \le 2^n-1}.
\end{align*}

We shall verify the stronger formulas:
\begin{align}
\label{a_B to a_L}
 a_B^{(0;n)}+a_B^{(1;n)} 
=
\sigma_n\left(a_L^{(0;n)}+a_L^{(1;n)} \right)\sigma_n^{-1}.
\end{align}
This is enough to conclude   proposition:
\begin{align*}
\label{transition BBS to LL}
M_B^{(n)} & =  a_B^{(0;n)}+a_B^{(1;n)} +a_B^{(0;n)*}+a_B^{(1;n)*} \\
& =
\sigma_n\left(a_L^{(0;n)}+a_L^{(1;n)} +a_L^{(0;n)*}+a_L^{(1;n)*}\right)\sigma_n^{-1}
= \sigma_n M_L^{(n)}
\sigma_n^{-1}.
\end{align*}

Since 
$\widehat\sigma_n$ is a permutation vector ({\normalfont Proposition \ref{prop permutation matrix}}) and $\widehat\sigma_n^{-1}=\widehat\sigma_n$ ({\normalfont Proposition \ref{prop sigma_n^2}}),
equation (\ref{a_B to a_L}) is equivalent to the equalities:
\begin{eqnarray}
\label{rel:alpha beta}
 \alpha^{(0)}_{j,k} +\alpha^{(1)}_{j,k} =
 \beta^{(0)}_{\widehat\sigma_n(j),\widehat\sigma_n(k)} +\beta^{(1)}_{\widehat\sigma_n(j),\widehat\sigma_n(k)}
\end{eqnarray}
for all $j,k \in I_n$.

Let us compute both sides, and
divide into  two cases on $(j_2+1+k_1)(j_3+1+k_2)\cdots (j_{n}+1+k_{n-1})$, where:
\begin{itemize}
 \item[(i)]  all the factors $j_{i+1}+1+k_{i}\, (i=1,2,\ldots,n-1)$ take odd integer values,
 \item[(ii)] otherwise, that is, there exists an integer $\kappa$ such that $j_{\kappa+1}+1+k_{\kappa}$ is even.
\end{itemize}

To treat  both cases,
we claim the following formula:
suppose  for some $1 \leq \kappa \leq n$,
 $$j_{i+1} + k_i +1$$
is odd for  $\kappa +1 \leq i \leq n-1$,
 then
\begin{eqnarray}
\label{useful_formula}
 [j_i'+k_i'+k_{i-1}'] = [j_i+k_{i-1}+1]
\end{eqnarray}
 for $\kappa+1 \leq i \leq n$.

In fact we have the equalities:
\begin{align*}
 [j_{i}'+k_{i}'+k_{i-1}']  &=
\bigg[3+\sum_{l=1}^{n-i+1} (j_{i+l-1}+k_{i+l-1})g_l^{(n-i+1)}+\sum_{l=1}^{n-i+2} k_{i+l-2}g_l^{(n-i+2)}\bigg]
\\&=
\bigg[3+j_{i}+k_{i-1}+{\!\!\!}\sum_{l=2}^{n-i+1}{\!\!} (j_{i+l-1}+k_{i+l-2})g_l^{(n-i+1)}
+2{\!\!}\sum_{l=2}^{n-i+1} {\!\!}k_{i+l+2} g_{l-1}^{(n-i+1)}+2k_n\bigg]
\\&=
[j_{i}+k_{i-1}+1]
\end{align*}
where we used the defining relation 
$g_m^{(n)} = g_{m-1}^{(n-1)} +g_m^{(n-1)}$.
This verifies the claim.

Case (i): 
From the assumption, we obtain
$ 
[j_{i+1}+k_{i}+1]=1$
 for $i \in \{1,2,\ldots,n-1\}$.
Thus 
one can show the equalities:
$$\alpha_{j,k}^{(0)}+\alpha_{j,k}^{(1)}
 =[\alpha_{j,k}^{(0)}]+[\alpha_{j,k}^{(1)}]=[j_1+1]+[j_1+2] =1.
$$
By use of (\ref{useful_formula}), we obtain the equalities:
\begin{align*}
\beta_{j',k'}^{(0)}+\beta_{j',k'}^{(1)} &
=[\beta_{j',k'}^{(0)}]+[\beta_{j',k'}^{(1)}] \\
& = ([j_1'+k_1']+[j_1'+k_1'+1])  [j_2+k_1+1][j_3+k_2+1] \cdots [j_n+k_{n-1}+1] \\
&=[j_1'+k_1']+[j_1'+k_1'+1] =1.
\end{align*}
Thus (\ref{rel:alpha beta}) is proven in this case.

Case (ii): 
In this case there exists the largest $\kappa$  such that  $j_{\kappa+1}+1+k_{\kappa}$ is even, then 
$\alpha_{j,k}^{(0)}$ and $\alpha_{j,k}^{(1)}$ are equal to 0.
Hence
\begin{align*}
 &\alpha_{j,k}^{(0)}+\alpha_{j,k}^{(1)} =0.
\end{align*}
Thus 
we obtain the equalities:
\begin{align*}
 \beta_{j',k'}^{(0)}+\beta_{j',k'}^{(1)}
 & = [\beta_{j',k'}^{(0)}]+[\beta_{j',k'}^{(1)}] \\
&  =([j_1'+k_1']+[j_1'+k_1'+1]) 
[j_2'+k_2'+k_1'] \cdots [j_{\kappa}' +k_{\kappa}' +k_{\kappa-1}'] \\
& \times 
[j_{\kappa+1}+k_{\kappa}+1] \cdots [j_{n+1}+k_n+1] \\
& =0
\end{align*}
since $[j_{\kappa+1}+k_{\kappa}+1] =0$.
This completes the proof.
\end{proof}

\vspace{3mm}

This research was supported by the Aihara Project, the FIRST program
from JSPS, initiated by CSTP, and by JSPS KAKENHI Grant Numbers 25400110.

\end{document}